\documentclass[12pt,a4paper]{amsart}
\usepackage{setspace}
\usepackage{hyperref}
\usepackage{amsfonts}
\usepackage[left=30mm,top=34mm,right=26mm,bottom=24mm]{geometry}
\usepackage{amssymb}
\usepackage{latexsym}
\usepackage{amsmath}
\usepackage{mathrsfs}
\usepackage{epigraph}
\usepackage[all,2cell]{xy}
 \UseAllTwocells \SilentMatrices
\usepackage{ifthen}
\usepackage{comment}

\usepackage{array,graphics,color}

\newenvironment{gap}
  {\color{blue}}%
  {}%

\newenvironment{nap} % nap - it's the new gap.
  {\color{red}}%
  {}%

\newtheorem{defn}{Definition}[section]
\newtheorem{lem}[defn]{Lemma}

\newtheorem{cor}[defn]{Corollary}
\newtheorem{ex}[defn]{Example}

\newtheorem{prop}[defn]{Proposition}

\newtheorem{rems}[defn]{Remarks}

\DeclareMathOperator{\Ker}{Ker}

\newcommand{\en}{E\ensuremath{{_0}}}

\title{CCR and CAR flows over convex cones
}

\author{R. Srinivasan}
\address{Chennai Mathematical Institute, H1, SIPCOT IT Park, Kelambakkam, Siruseri 603103, India.}
\email{vasanth@cmi.ac.in}


\begin{document}
%%% addition formula for tensor products

%%%%% NEW COMMANDS %%%%%

%standard maths
\newcommand{\norm}[1]{\ensuremath{\left\|#1\right\|}}
\newcommand{\ip}[1]{\ensuremath{\left\langle#1\right\rangle}}
\newcommand{\dist}{\hbox{dist}}
\newcommand{\bra}[1]{\ensuremath{\left\langle#1\right|}}
\newcommand{\ket}[1]{\ensuremath{\left|#1\right\rangle}}
\newcommand{\lin}{\ensuremath{\mathrm{Span}}}
\renewcommand{\ker}{\ensuremath{\mathrm{ker}}}
\newcommand{\ran}{\ensuremath{\mathrm{Ran}}}
\newcommand{\dom}{\ensuremath{\mathrm{Dom}}}
\newcommand{\supp}{\ensuremath{\mathrm{supp}}}
\newcommand{\id}{\hbox{id}}

%tensor products
\newcommand{\overtimes}{\ensuremath{\overline{\otimes}}}
\newcommand{\undertimes}{\ensuremath{\underline{\otimes}}}
\newcommand{\opower}[1]{\ensuremath{^{\otimes #1}}}

%special letters
\newcommand{\N}{\ensuremath{\mathbb{N}}} %naturals
\newcommand{\Z}{\ensuremath{\mathbb{Z}}} %integers
\newcommand{\Q}{\ensuremath{\mathbb{Q}}} %rationals
\newcommand{\R}{\ensuremath{\mathbb{R}}} %reals
\newcommand{\C}{\ensuremath{\mathbb{C}}} %complex
\newcommand{\F}{\ensuremath{\mathcal{F}}} %Fock space
\newcommand{\B}{\ensuremath{\mathcal{B}}}
\newcommand{\D}{\ensuremath{\mathcal{D}}} %a domain
\newcommand{\E}{\ensuremath{\mathcal{E}}} %exponential vectors
\newcommand{\W}{\ensuremath{\mathcal{W}}}
\newcommand{\V}{\ensuremath{\mathcal{V}}}
\renewcommand{\H}{\ensuremath{\mathcal{H}}}
\newcommand{\K}{\ensuremath{\mathsf{K}}}
\newcommand{\h}{\ensuremath{\mathrm{h}}}
\renewcommand{\k}{\ensuremath{\mathrm{k}}}
\newcommand{\J}{\ensuremath{\mathscr{J}}}
\newcommand{\A}{\ensuremath{\mathcal{A}}}
\renewcommand{\L}{\ensuremath{\mathcal{L}}}
\newcommand{\n}{\ensuremath{\mathrm{N}}} %von Neumann algebra
\newcommand{\m}{\ensuremath{\mathrm{M}}} %von Neumann algebra

%E-semigroups
\newcommand{\munit}{\ensuremath{\mu\hbox{nit~}}}
\newcommand{\munits}{\ensuremath{\mu\hbox{nits~}}}
\newcommand{\unitset}[1]{\ensuremath{\mathcal{U}_{#1}}}
\newcommand{\unitspace}[1]{\ensuremath{H(\mathcal{U}_{#1})}}
\newcommand{\munitset}[1]{\ensuremath{\mathcal{U}_{#1,#1'}}}
\newcommand{\munitspace}[1]{\ensuremath{H(\mathcal{U}_{#1,#1'})}}
\newcommand{\gaugespace}[1]{\ensuremath{H(G(#1))}}
\newcommand{\ind}{\ensuremath{\mathrm{Ind}}}

%Stochastic calculus
\newcommand{\expectation}{\ensuremath{\mathbb{E}}}
\newcommand{\Exp}{\ensuremath{\mathrm{Exp}}}
\newcommand{\Log}{\ensuremath{\mathrm{Log}}}

%Operator algebras
\newcommand{\alg}{\ensuremath{\mathrm{A}}}
\newcommand{\factor}{\ensuremath{\mathrm{M}}}
\newcommand{\oneinfty}{I$_\infty~$}
\newcommand{\twoone}{II$_1~$}
\newcommand{\twoinfty}{II$_\infty~$}
\newcommand{\three}[1]{III$_{#1}~$}
\newcommand{\hyperfinite}{\ensuremath{\mathcal{R}}}
\newcommand{\semiflowalg}{\ensuremath{\mathcal{A}}}

\begin{abstract}
Recently it is proved in \cite{SS} that CCR flows over convex cones are cocycle conjugate if and only if the associated isometric representations are conjugate. We provide a very short, simple and direct proof of that. Using the same idea we prove the analogous statement for CAR flows as well. Further we show that CCR flows are not cocycle conjugate to the CAR flows when the (multi-parameter) isometric representation is  `proper', a condition which  is satisfied by all known examples.
\end{abstract}

 \maketitle

\section{Introduction}
The study of one parameter $E_0-$semigroups, initiated by R.T. Powers, with enormous contribution by Arveson, has been an active area of research in operator algebras, for more than thirty years. In last twenty years there are also some exciting works on type II and type III $E_0-$semigroups, and on $E_0-$semigroups on non-type-I factors. Very recently  $E_0-$semigroups over convex cones are being investigated, leading to  some interesting results (see \cite{anbu}, \cite{anbu-SS}).

Arveson introduced the notion of decomposability for product systems of one parameter $E_0-$semigroups. He showed in \cite{arv-d} that one parameter $E_0-$semigroups associated with decomposable product systems are cocycle conjugate to CCR flows. In \cite{SS},  some of Arveson's ideas are  extended to the multi-parameter settings to prove the injectivity of CCR functor.  But the proof is very long, runs in to several pages involving messy computations. We provide a very short proof, which follows almost immediately after the definitions.  Also the ideas of \cite{SS} does not help to tackle the case of CAR flows. Even for $1-$parameter CAR flows it is difficult to describe the decomposable vectors, and the result is not available in the literature.  In fact even the units of $1-$parameter CAR flows are described in \cite{MS1} through an indirect way. Our ideas works perfectly for the CAR case also, establishing the injectivity of CAR functor also, with little bit more effort. 

We finally investigate the cocycle conjugacy between CCR flows and CAR flows. In contrast to the $1-$parameter case we show that they are not cocycle conjugate in in most of the cases. We assume a technical condition which is satisfied by all our known examples.

This version will be replaced by an updated version soon. Part of this work was done when the author was visiting University of Kyoto, during April to November 2018. The author would like to acknowledge the JSPS fellowship and thank Masaki Izumi for inviting to University of Kyoto.  The author also would like to thank Masaki Izumi for discussions.

\section{preliminaries}\label{pre}

We denote the real numbers, complex numbers and natural numbers by  $\R, \C$ and $ \N$ respectively and  by $\R_+$ we denote the semigroup $(0,\infty)$. For $k \in \N$ we denote $[k] = \{1, 2, \cdots k\}$. 

 All our Hilbert spaces are complex, separable, and are  equipped with inner products which are  anti-linear in the first variable and linear in the second variable. 
  For a measurable set $S$ and a Hilbert space $\k$ of finite or infinite dimension, $L^2(S,\k)$ denote the Hilbert space of all square integrable functions from $S$ to $\k$.  The $1-$parameter semigroups of right shifts $\{S_t: t \in \R_+\}$  on  $L^2(\R_+,\k)$ is defined by 
 \begin{eqnarray*}(S_tf)(s) & = & 0, \quad s<t,\\
& = & f(s-t), \quad s \geq t,
\end{eqnarray*} 
for $f \in L^2(\R_+,\k)$.

Let $P \subset \mathbb{R}^{d}$ be a closed convex cone. We assume that $P$ is  spanning  and pointed, i.e. $P-P=\mathbb{R}^{d}$  and $P \cap -P=\{0\}$. 
Let $\Omega$ denote the interior of $P$. Then $\Omega$ is dense in $P$.    Further  $\Omega$ is also spanning i.e. $\Omega-\Omega=\mathbb{R}^{d}$.  For $x, y \in \mathbb{R}^{d}$, we write $x \geq y$ and $x > y$ if $x-y \in P$ and $x-y \in \Omega$ respectively.

 \begin{defn}
 An $E_0-$semigroup over $P$  on $B(H)$ is a family of normal unital $*-$endomorphisms $\alpha = \{\alpha_{x}: x \in P\}$  of $B(H)$ satisfying

\begin{itemize}
\item[(i)] $\alpha_{x}\circ \alpha_{y}=\alpha_{x+y}$ for all $x, y \in P$ and  $\alpha_0= I_{B(H)}$,
\item[(ii)]  the map $P  \ni x \to \ip{\alpha_{x}(A)\xi,\eta} \in \mathbb{C}$ is continuous for all $A \in B(H)$ and $\xi,\eta \in H$. \end{itemize} 

An $E_0-$-semigroup $\alpha$ is  said to be pure if $\cap_{t\geq 0} \alpha_{tx}(\m) =\C, ~~~\forall x\in \Omega$  
\end{defn}

 In this article we always assume that our $E_0-$semigroups are pure, without mentioning.

\begin{defn}
Let $\alpha:=\{\alpha_{x}: x \in P\}$ be an $E_0-$semigroup on $B(H)$.
An $\alpha$-cocycle is  a strongly continuous family of unitaries $\{U_{x}\}_{x \in P}$ satisfying $$U_{x}\alpha_{x}(U_y)=U_{x+y}~~\forall x,y\in P.$$
A cocycle  $\{U_{x}\}_{x \in P}$ is said to be a gauge cocycle if further $U_x \in \alpha_x(B(H))^\prime$.
\end{defn}

Given an $\alpha$-cocycle $\{U_{x}\}_{x \in P}$, it is easy to  verify that 
 $\{Ad(U_x)\circ \alpha_{x}: x \in P \}$ is also an $\en$-semigroup on $B(H)$. This is called  the cocycle perturbation of $\alpha$ by the cocycle $\{U_{x}\}_{x \in P}$.
 
\begin{defn}  Let $\alpha:=\{\alpha_{x}: x \in P\}$ and $\beta:=\{\beta_{x}:  x \in P\} $ be two $\en$-semigroups on $B(H)$ and $B(K)$ respectively. We say that

 \begin{itemize}
  \item[(i)]  $\alpha$ is  conjugate to $\beta$ if there exists a unitary operator $U:H \to K$ such that for every $x \in P$, $\beta_{x}=Ad(U)\circ \alpha_{x} \circ Ad(U^{*})$, and
   \item[(ii)]  $\alpha$ is cocycle conjugate to $\beta$ if there exists a unitary $U:H \to K$ such that  the $\en$-semigroup $\{Ad(U) \circ \alpha_{x} \circ Ad(U)^{*}\}_{x \in P}$ is a cocycle perturbation of $\beta$.  
\end{itemize}
\end{defn}

Clearly cocycle conjugacy is an equivalence relation.   

 For a  complex separable Hilbert space $K$, we denote the symmetric Fock space by $\Gamma_s(K) ):=\bigoplus_{n=0}^\infty{(K)^{\vee n}}$, with vacuum vector $\Omega_s$. 
 We refer to \cite{KRP} for  proofs of the following well-known facts.  
 For $u \in K$, the exponential of $u$ is defined by  $$e(u):=\sum_{n=0}^{\infty}\frac{u^{\otimes n}}{\sqrt{n!}}.$$ Then the set $\{e(u): u \in K\}$ is linearly independent and total in $\Gamma_s(K)$. Exponential vectors satisfy $\ip{e(u),  e(v)} = e^{\ip{u, v}}$ for all $u,v\in K.$   
For $u \in K$, there exists a  unitary operator, denoted $W(u)$ on $\Gamma_s(K)$ determined uniquely by the equation $$W(u)e(v):=e^{-\frac{\|u\|^{2}}{2}-\ip{u|v}}e(u+v), ~~ \forall  v\in K.$$
The operators $\{W(u): u \in K\}$ are called the Weyl operators and they  satisfy the well-known  canonical commutation relations:  
\[W(u)W(v)=e^{-i Im (\ip{u,v)}}W(u+v)~~ \forall u, v\in K.\]
 Further the linear span of $\{W(u): u \in K\}$ is a strongly dense unital $*$-subalgebra of $B(\Gamma_s(K))$. 
 For an isometry $U:K_1 \to K_2$,  its Bosonic second quantization is the isometric operator $\Gamma_s(U): \Gamma_s(K_1) \to \Gamma_s(K_2)$,  satisfying $\Gamma_s(U)e(v)=e(Uv)$ for all $v \in K_1$.   Second quantized unitaries are related to Weyl operators by  $\Gamma(U)W(v)\Gamma(U)^{*}=W(Uv)~~ \forall v \in K_1.$   

Let $\Gamma_a(K):=\bigoplus_{n=0}^\infty{(K)^{\wedge n}}$ be the antisymmetric Fock space over $K$, 
%i.e. the sum of antisymmetric tensor powers of $K$, 
with vacuum vector $\Omega_a$.
For any $f\in K$ the Fermionic creation operator $a^*(f)$ is the bounded operator defined by the linear extension of
$$a^*(f)\xi=\left\{\begin{array}{ll} f & \hbox{if}~\xi=\Omega, \\ f\wedge \xi &  \hbox{if}~ \xi\perp\Omega, \end{array}\right.$$ where $\Omega$ is the vacuum vector. 
%($1$ in the 0-particle space $\C$), and $f\wedge\xi$ is the antisymmetric tensor product. 
The annihilation operator is defined by the adjoint $a(f)=a^*(f)^*$. The creation and annihilation operators satisfy the well-known anti-commutation relations  and they  generate $B(\Gamma_a(K))$ as a von Neumann algebra. For an isometry $U:K_1 \to K_2$,  its Fermionic second quantization is the isometric operator $\Gamma_a(U): \Gamma_a(K_1) \to \Gamma_a(K_2)$,  satisfying 
$\Gamma_a(U)(\xi_1 \wedge \cdots \wedge \xi_n)= U\xi_1 \wedge \cdots \wedge U\xi_n$ for all $\xi_i \in K_1$, $i\in [n]$.

The basic examples of $E_0-$semigroups are CCR flows and CAR flows associated with an isometric representation.  By an isometric representation we mean a strongly continuous semigroup of isometries indexed by $P$.

\begin{ex}\label{ex-CCR-CAR}  Let $K$ be a separable Hilbert space and $V:P \to B(K)$ be an isometric representation of $P$ on $K$. 
Then CCR and CAR flows are constructed as follows.

\begin{itemize}
\item[(i)] (CCR flows) Let $U^s_x: \Gamma_s(\Ker(V_x^*)) \otimes \Gamma_s(K) \mapsto \Gamma_s(K)$   be the extension of \begin{equation}\label{UCCR}  e(\xi_x) \otimes e(\xi) \mapsto e(\xi_x \oplus V_x \xi)~~ \forall \xi_x \in \Ker(V_x^*), \xi \in K. \end{equation} Define $\alpha_x(X) =  U^s_x(1_{\Gamma_s(\Ker(V_x^*))}\otimes  X )(U^s_x)^*$.
Then $\alpha^{V}:=\{\alpha^V_{x}\}_{x \in  P}$  is the unique $E_0$-semigroup (see \cite[Proposition 4.7]{anbu}) on $B(\Gamma(K))$ satisfying \[\alpha^V_{x}(W(u))=W(V_{x}u)~~ \forall ~ x \in P,~u \in K.\]

\item[(ii)]  (CAR flows)  Let  $U^a_x: \Gamma_a(\Ker(V_x^*)) \otimes \Gamma_a(K) \mapsto \Gamma_a(K)$  be the extension of 
\begin{equation}\label{UCAR}(\xi_1 \wedge\xi_2  \cdots \wedge \xi_{m}) \otimes (\eta_1 \wedge\eta_2   \cdots \wedge \eta_{n}) \mapsto V_x \eta_1 \wedge V_x \eta_2 \cdots \wedge V_x \eta_{n} \wedge \xi_1 \wedge\xi_2   \cdots \wedge \xi_{m},\end{equation} for $\xi_i \in \Ker(V_x^*)$, $\eta_j \in K$, and $i\in [m], j \in [n]$. Define $$\beta_x(X) =  U^a_x(1_{\Gamma_a(\Ker(V_x^*))}\otimes  X )(U^a_x)^*.$$
Then it can be verified, as in the case of CCR that  $\beta^{V}:=\{\beta^V_{x}\}_{x \in P}$ is the unique $E_0$-semigroup on $B(\Gamma_a(K))$ satisfying \[
\beta^V_{x}(a(u))=a(V_{x}u)~~ \forall ~ x \in P,~u \in K.\]
 \end{itemize}

The CCR flow $\alpha$ and the CAR flow $\beta$ are pure $E_0-$semigroups if the isometric representation $V$ is pure, that is   $\cap_{t\geq 0} V_{ta} (K) = \{0\}$ for all $a \in \Omega$. We assume that all the isometric representations under consideration in this article are pure.
\end{ex}

For an $E_0-$-semigroup $\alpha$ on $B(H)$ we can associate a concrete product system of Hilbert spaces $\E^\alpha =\{E^\alpha_x: x \in P\}$  by $$E^\alpha_x = \{X \in B(H): \alpha_x(T)X = XT~\forall T \in B(H)\},$$ with inner product  $T^*S = \ip{T, S}1_H$.  Also the map $U_{x,y}: E^\alpha_x\otimes E^\alpha_y \mapsto E^\alpha_{x+y}$ defined by $T_x \otimes T_y \mapsto T_xT_y$ is unitary.  We denote the product system by $(E^\alpha_x, U_{x,y})$

An isomorphism between product systems $(E_x, U_{x,y})$ and $(E'_x, U'_{x,y})$ is a family $\{W_x:  x \in P\}$ of unitary operators between $E_x\mapsto E'_x$ satisfying 
$$ U'_{x+y} (W_x \otimes W_y) = W_{x+y} U_{x,y}, ~~ \forall x,y \in P.$$
There are some measurability conditions involved in the definition of both product system and its isomorphisms, which we assume, but we do not state here, since they are not used explicitly in any of the proofs. The interested reader can refer to  \cite{murugan}. Concrete product systems form a complete invariant for $E_0-$semigroups up to the cocycle conjugacy equivalence (see \cite{murugan}).  

The product systems associated of CCR flows and CAR flows, associated with an isometric $\{V_x: x \in P\}$ on $K$, can be described as follows. Let $ U^s_x,  U^a_x$ be the maps as defined in (\ref{UCCR}) and (\ref{UCAR}) respectively. Let $$E^s_x = \Gamma_s(\Ker(V_x^*)); E^a_x = \Gamma_a(\Ker(V_x^*)),$$ and by defining $T^s_{\xi_x}(\xi) = U^s_x(\xi_x \otimes \xi)$, $T^a_{\eta_x}(\eta) = U^a_x(\eta_x \otimes \eta)$ for $\xi_x \in E^s_x, \xi \in \Gamma_s(K)$, $\eta_x \in E^a_x, \eta \in \Gamma_a(K)$, we can identify the product systems of CCR and CAR flows with $(E_x^s, U^s_{x,y})$ and $(E_x^a, U^a_{x,y})$ respectively.   The product map 
$U^s_{x, y}: E_x^s\otimes E_x^s   \mapsto E_{x+y}^s$, $U^a_{x,y}: E_x^a\otimes E_y^a
\mapsto E_{x+y}^a,$ are defined in a similar way as $U^s_x$, $U^a_x$, by replacing $K$ with $\Ker(V_y^*)$ appropriately in (\ref{UCCR}) and (\ref{UCAR}) respectively. 

While defining units for $E_0-$-semigroups over general convex cones, we need to take the cohomology into consideration. But, since we deal only with CCR flows and CAR flows, which always admit a canonical unit 
belonging to  the trivial cohomological class, the following definition suffices for our purposes in this article. 

\begin{defn}  Let $\alpha:=\{\alpha_{x}: x \in P\}$ be an $E_0-$semigroup on $B(H)$ A unit for $\alpha$ is a strongly continuous family $u =\{u_x : x \in P\}\subseteq B(H)$, satisfying $u_x \in E^\alpha_x$ and $u_x u_y = u_{x+y}$ for all $x,y  \in P$. (In the language of product systems this means $U_{x,y}(u_x\otimes u_y) =u_{x+y}$.)
\end{defn}

An $E_0-$semigroup (or equivalently the associated product system)  is called spatial if it admits a unit.  In spatial product systems we fix a special $\Omega= \{\Omega_x : x \in P\}$, with $\|\Omega_x\|=1$ for all $x \in P$,   as the canonical unit. For CCR flows and CAR flows the canonical unit  is given by the vacuum vectors $\Omega^s=\{\Omega^s_x\in \Gamma_s(\Ker(V_x^*)) : x \in P\} $ and  $\Omega^a=\{\Omega^a_x  \in \Gamma_a(\Ker(V_x^*)): x \in P\}$ respectively. As operators these are represented by the  second quantization of the isometric representation $\{\Gamma_s(V_x): x \in P\}$ and  $\{\Gamma_a(V_x): x \in P\}$ respectively.  

\begin{defn}
An $E_0-$semigroup is said to be of type I if units exists and generate the product system, that is products of the form $u^1_{x_1} u^2_{x_2} \cdots u^n_{x_n}$, with each $u^i$ a unit and $x_1+x_2\cdots x_n= x$ are total in $E_x$. 
\end{defn}

An exponential unit is a unit $u$ satisfying $\ip{u_x,\Omega_x}=1$ for all $x \in P$. 
We denote 
the collection of all exponential units of an  $E_0-$-semigroup $\alpha$ by $\mathfrak{U}_\Omega(\alpha)$.  The set of all exponential units of CCR flows are described by the additive cocycles of the associated  isometric representation. 

\begin{defn}
Let $V: P \to B(K)$ be an isometric representation. A continuous function $h : P \to K$ is called an additive cocycle  for $V$  if  $h_x \in \Ker(V_x^*)$ for all $x \in P$ and  $h_x+ V_x h_y = h_{x+y}.$ for all $ x, y \in P$.
\end{defn}

We denote the collection of all additive cocycle of an isometric representation $V$ by $\mathfrak{A}(V)$.   When $\alpha$ is the CCR flow associated with $V$, the map $$\{h_x: x \in P\}\mapsto \{e(h_x)\: x \in P\}$$  provides a bijection between  $\mathfrak{A}(V)$ and $\mathfrak{U}_\Omega(\alpha)$ (see \cite[Theorem 5.10]{anbu}).

\section{additive decomposability}   In this Section we define an additive version of decomposability for spatial product systems and construct an isometric representation, which forms a cocyce conjugacy invariant for the associated $E_0-$semigroups, under some conditions.  Throughout this section we fix an (arbitrary) spatial $E_0-$semigroup $\alpha=\{\alpha_x: x \in P\}$  with canonical unit $\Omega=\{\Omega_x: x \in P\}$ and product system $E = \{E_x: x \in P\}$.

\begin{defn}  An element $a \in E_x$ is said to be additive decomposable if for all $y\leq x \in P$ there exists $a_y \in E_{y}, a_{x-y}\in E_{x-y}$ satisfying $a_x \perp \Omega_x$,  $a_{x-y} \perp \Omega_{x-y}$ and  $$U_{y, x-y}(a_y\otimes \Omega_{x-y}) + U_{y, x-y}(\Omega_{x} \otimes a_{x-y}) = a.$$
\end{defn}

%comment on arveson decomposability and integration with resepct to martingales 

We denote by $D_x(\alpha)$ the set of all decomposable vectors in $E_x$.  Clearly $D_x(\alpha)$ is a vector space. Since $U_{y, x-y}(a_y\otimes \Omega_{x-y})$ is orthogonal to $U_{y, x-y}(\Omega_{x} \otimes a_{x-y})$ in the above decomposition, it also follows that $D_x(\alpha)$ is a Hilbert space with respect to the restricted inner product.

\begin{lem}\label{ind-d} Let $a\in D_{x+y}$ be such that $U_{x, y}(a_x\otimes \Omega_{y}) + U_{x, y}(\Omega_{x} \otimes a_{y}) = a$, with $E_x\ni a_x\perp \Omega_x, E_y \ni a_y\perp \Omega_y$, then $a_x \in D_x, a_y \in D_y$.
\end{lem}

\begin{proof} Let $z\leq x$. Since $a \in D_x$, there exist $E_z\ni a_z\perp \Omega_z,  E_{x+y-z} \ni a_{x+y-z} \perp \Omega_{x+y-z} $, satisfying $U_{z, x+y-z}(a_z\otimes \Omega_{x+y-z}) + U_{z, x+y-z}(\Omega_{z} \otimes a_{x+y-z}) = a$. This combined with the decomposition in the statement of the lemma, imply that $$(a_x - (a_z \otimes \Omega_{x-z})) \otimes \Omega_y = \Omega_z \otimes (a_{x+y-z} - (\Omega_{x-z}\otimes a_y)).$$ This forces $a_{x+y-z} - (\Omega_{x-z}\otimes a_y)\in E_{x-z} \otimes \Omega_y$. So  there exists an $E_{x-z} \ni a_{x-z} \perp \Omega_{x-z}$ satisfying $U_{z, x-z}(a_z\otimes \Omega_{x-z}) + U_{z, x-z}(\Omega_{z} \otimes a_{x-z}) = a_x$.
\end{proof}

We will be assuming the following embeddability assumption, which is satisfied by both CCR flows and CAR flows.

\begin{defn}  A spatial product system is said to be embeddable if $$U_{x, y}(D_x\otimes \Omega_{y})\subseteq D_{x+y}, ~~U_{x, y}(\Omega_{x} \otimes D_{y}))\subseteq D_{x+y} \forall x,y \in P.$$
\end{defn}

For an embeddable product system,  the embeddings  $$\iota_{x,x+y}:D_x  \mapsto D_{x+y} \qquad \xi\mapsto U_{x,y}(\xi\otimes\Omega_{y})$$ allow us to construct an inductive limit of the family of Hilbert spaces $\{D_x\}_{x \in P}$, which we denote by $D_\infty$, together with embeddings  $\iota_x:D_x\to D_\infty$.  
Notice that each of the vector $\Omega_x$ is mapped to the same element, which we denote by $\Omega_\infty\in D_\infty$.
 We can also define a second family of embeddings
 $$\kappa_{x,y}:D_x\mapsto D_{x+y} \qquad \xi\mapsto U_{x,y}(\Omega_x\otimes\xi).$$ Thanks to the associativity axiom,  the squares
 $$ \xymatrix{ D_x \ar[r]^{\iota_{x,y}} \ar[d]_{\kappa_{z,x}} & D_{x+y} \ar[d]^{\kappa_{z,x+y}} \\ D_{z+x} \ar[r]_{\iota_{z+x,y}} & D_{z+x+y}} $$
 commute for all $z,x,y\in P$. So there exist isometries $(\kappa_x:D_\infty\mapsto D_\infty)_{x\in P}$, which defines an isometric representation $\kappa =\{\kappa_x: x \in P\}$ of $P$ on $D_\infty$, satisfying
 $\kappa_x\iota_y=\iota_{x+y}\kappa_{x,y}$.  
 
Any isomorphism between two spatial embeddable product systems, fixing the canonical unit, will map additive decomposable vectors to additive decomposable vectors. Consequently it will induce a unitary map between the Hilbert spaces constructed above conjugating the isometric representations. When the gauge cocycles acts transitively on the set of all units, we can replace the given isomorphism with an isomorphism which fixes the canonical units. In that case the tuple $(D_\infty, \kappa)$ is a cocycle conjugacy invariant for the associated $E_0-$semigroup.  This holds in particular for examples with  the canonical unit being the only unit up to scalars, which is the case in most of our examples.

\section{ The injectivity of CCR and CAR functors} Throughout this section we fix an isometric representation $V$ on $K$. Let $\alpha$ and $\beta$ be the CCR flow and CAR flow respectively, associated with $V$. 
As in Section \ref{pre}, we denote $E^s_x =\Gamma_s(\Ker(V_x^*))$ and $E^a_x =\Gamma_a(\Ker(V_x^*))$. 
We also just denote by $E_x =\Gamma(\Ker(V_x^*))$ referring to the product systems of both CCR and CAR, to make statements on both of them at the same time. Similarly the second quantization $\Gamma(V_x)$ denotes both $\Gamma_s(V_x)$ and $\Gamma_a(V_x)$.

\begin{prop}\label{H} For both $\alpha$ and $\beta$  we have $D_x(\alpha)=\Ker(V_x^*)$ and $D_x(\beta)=\Ker(V_x^*)$ embedded as the $1-$particle space in $E^s_x$ and $E^a_x$ respectively. 

Consequently the product systems of both $\alpha$ and $\beta$ are embeddable with the invariant $(D_\infty, \kappa)$ coinciding with $(K,V)$.
\end{prop}

\begin{proof} It is clear from the product map of the product systems, that the  additive decomposable vectors  are those $\xi \in \Gamma(\Ker(V_{x+y}^*))$ for which there exists  $\xi_x \in \Gamma(\Ker(V_{x}^*))$  and $\xi_y \in \Gamma(\Ker(V_{y}^*))$ satisfying $\xi_{x+y} = \xi_x + \Gamma(V_x)\xi_y$ ($ \Gamma(\Ker(V_{x}^*))$ and  $\Gamma(\Ker(V_{y}^*))$ embedded into $\Gamma(\Ker(V_{x+y}^*))$ using $\Omega$).  Since $\Gamma(V_x)$ leaves the $n-$particle spaces invariant, we only need to look for $n-$particle vectors satisfying the additive decomposability condition. 

Since any $\xi \in \Ker(V_{x+y}^*)$ can be decomposed as $\xi= \xi_x + V_x \xi_y$ for some  $\xi_x \in  \Ker(V_{x})^*$ and $\xi_y \in  \Ker(V_{y})^*$, 
the $1-$particle space $\Ker(V_x^*)$ is contained in  $D_x(\alpha)$ and  $D_x(\beta)$. So we only have to prove there does not exist any non-zero  additive decomposable vectors in the $n-$particle spaces for $n >1$. Let us first show this for a $1-$parameter CCR/CAR flow.  

We can assume, without loss of generality, that the associated isometric representation is the standard right shift $\{S_t\}_{t\in \R_+}$ on $L^2(\R_+, \k)$.   
Let $n\geq 2$. A vector in the $n-$particle space of $\Gamma(L^2((0,t),\k))$ can be identified with a function 
$F  \in  L^2((0,t)^n, \k\opower{n})$. 
For any $k\in\N$, thanks to Lemma \ref{ind-d} and by induction,
$$F=\sum_{i=0}^{2^k-1} \Gamma(S_{2^{-k}it}) F_i$$ for some $F_i \in  L^2((0,2^{-k}t)^n, \k\opower{n})$.
Thus, modulo a null set,	
$$\supp( F) \subset \bigcup_{i=0}^{2^k-1}[2^{-k}it,2^{-k}(i+1)t]^{\times n}.$$
Since this holds for every $k\geq 0$ we get $\supp (F) \subset \{x\in\R_+^n:~x_1=x_2=\ldots=x_n\},$ which has measure zero.

Now if any $\xi \in E_x$ is decomposable, it is decomposable with respect to $\{E_{tx}\}_{t\geq 0}$, which implies that $\xi$ is in the $1-$particle space.  $1-$particle spaces are mapped into $1-$particle spaces under the product map of product system. Also since the isometric representation is pure, the inductive limit of $\Ker(V_x)^*$ is $K$ itself. The rest of the statements are clear now
\end{proof}

As mentioned in the end of Section \ref{pre} the units of the CCR flow $\alpha$  are indexed by the additive cocycles of $V$, up to scalars. But any additive cocycle $\{h_x: x \in P\}$ also gives rise to a gauge cocycle $\{W(h_x): x \in P\}$. So the gauge cocycles of CCR flows acts transitively on the units. Now the following corollary is immediate from the  above proposition and the discussion in the previous section.

\begin{cor}  The CCR flows associated with isometric representations $V^1$ and $V^2$ are cocycle conjugate if and only if $V^1$ and $V^2$ are conjugate. In that case the CCR flows are actually conjugate. 
\end{cor}

The proof for CAR flows is slightly involved. In the next proposition we show that units of CAR flows also arise from additive cocycles.  We use the bijection between units and centered addits  established in  \cite[Section 5]{MS1}, for $1-$parameter systems.  (Notice,  in the case of CCR/CAR, that the  centered addits are just the additive cocycles of $V$ sitting in the $1-$particle space.)  We cite precise theorems and refer the reader to \cite{MS1}.

\begin{prop}\label{CAR-unit}  There exists an injective map $\Log_\Omega:  \mathfrak{U}_\Omega(\beta) \mapsto \mathfrak{A}(V)$ satisfying  $$ \ip{u_x, u'_x} = e^{\ip{\Log_\Omega(u)_x,\Log_\Omega(u')_x}}~~~\forall x \in P.$$
\end{prop}

\begin{proof}
Given a unit $u \in  \mathfrak{U}_\Omega(\beta)$, we fix an arbitrary $x \in P$, and consider $\{u_{tx}\}_{t\in \R_+}$ a unit for $\{\beta_{tx}\}_{t\in \R_+}$. Then, thanks to  \cite[Proposition 5.10]{MS1} there exists an additive cocycle $\{\Log_\Omega(u)_{tx}\}_{t\in \R_+}$ of $\{V_{tx}\}_{t\in\R_+}$. By setting $t=1$, the family $\{\Log_\Omega(u)_x: x \in P\}$ is well-defined for each $u \in \mathfrak{U}_\Omega(\beta)$, and  for $u, u' \in \mathfrak{U}_\Omega(\beta)$ it satisfies $ \ip{u_x, u'_x} = e^{\ip{\Log_\Omega(u)_x,\Log_\Omega(u')_x}}$ for all $x \in P$ (see \cite[Proposition  5.9, 5.10, 5.11]{MS1}).  If $\Log_\Omega(u) = \Log_\Omega(u')$ then by applying the $\Exp$ map  (see \cite[Propositions 5.9, 5.11]{MS1}) on  $\{\Log_\Omega(u)_{tx}\}_{t\in \R_+}$   and $\{\Log_\Omega(u')_{tx}\}_{t\in \R_+}$ we get $u_x=u'_x$ for all $x\in P$.  We are only left to prove $\{\Log_\Omega(u)_x: x \in P\}$ is an additive cocycle for $V$.

The unit $u$ admits an orthogonal summation $u_x = \Omega_x +\sum_{n=1}^\infty A^n_x$, where $A^1_x = \Log_\Omega(u)_x$ and $A^{n+1}_x =\int_0^1 A^n_{tx} d\Log_\Omega(u)_{tx}$ obtained by iterating the  It\^o integration (see the first paragraph of proof of \cite[Proposition 5.9]{MS1}).   It is clear from the definition of  It\^o integration  in  the beginning of \cite[Section 5]{MS1},  that $A^n_x$ belongs to the $n-$particle space. So it follows that $ \Log_\Omega(u)_x$ is projection of $u_x$ onto the $1-$particle space. Now the relation $u_x\otimes u_y=u_{x+y}$ for all $x,y \in P$ implies that $ \Log_\Omega(u)_x$ satisfies the additive cocycle condition in $x,y$.
\end{proof}

Unlike the CCR flows, an arbitrary additive cocycle of $V$ need not give rise to a unit for the CAR flows, as remarked in  \ref{nounit}

\begin{defn}
An isometric representation $V$ is said to be divisible if  $$\overline{span} \{V_x h_y: h \in  \mathfrak{A}(V), x,y \leq z\} =\Ker(V_{z}^*) ~~\forall  z\in P.$$ 
\end{defn}

\begin{rems}\label{nounit}
Since the set of units of a CCR flow are in bijection with additive cocycles, a CCR flow is of type I if and only if the associated isometric representation is divisible. 

A basic family of  examples of divisible isometric representations, when $P =\R_+^2$, is $V_{s,t} =S_s\oplus S_t$ on $L^2((0,\infty),\k_1)\oplus L^2((0,\infty),\k_2)$. In this case the additive cocycles are  $$h_{s,t} =(1_{0,s}\otimes \xi_1) \oplus  (1_{0,t}\otimes \xi_2), ~~\xi_1 \in \k_1, \xi_2 \in \k_2.$$ It is easy to check that $V$ is divisible, and hence the CCR flow $\alpha$ is of type I.
 On the other hand the corresponding CAR flow $\beta$ is not type I, which follows from the following Proposition \ref{I} and the non-conjugacy proved in Proposition \ref{non-conj}. 
 \end{rems}

\begin{prop}\label{I}
If the CAR flow $\beta$ is type I then it is cocycle conjugate to the CCR flow $\alpha$ of the same isometric representation.
\end{prop}

\begin{proof}
Since the CAR flow is of type I, the set $\{V_x \Log_\Omega(u)_y: u \in  \mathfrak{U}_\Omega(\beta), x+y \leq z\}$  is total in $\Ker(V_z^*)$ for all $z \in P$. If not, since the units are obtained by the $\Exp$ map of \cite[Proposition 5.9]{MS1}, all of them are contained in the product system generated by $\{V_x \Log_\Omega(u)_y: u \in  \mathfrak{U}_\Omega(\beta), x+y \leq z\}$ which will be  a proper subspace of $\Gamma_a(\Ker(V_z^*))$. So the units can not generate the product system of $\beta$.

Define a map from $\Gamma_a(\Ker(V_{x}^*))$ to $\Gamma_s(\Ker(V_{x}^*))$  by 
$$u^1_{x_1} u^2_{x_2} \cdots u^n_{x_n} \mapsto  e(\Log_\Omega(u^1)_{x_1})  e(\Log_\Omega(u^2)_{x_2} ) \cdots  e(\Log_\Omega(u^n)_{x_n}).$$ 
Thanks to Proposition \ref{CAR-unit} this map preserves inner product. Also thanks to type I property of $\beta$ and the divisibility of $V$, this map maps total set of vectors onto a total set of vectors. Hence it extends to an unitary map, which can easily be seen to provide an isomorphism between the product systems. 
\end{proof}
 
Now we prove the injectivity of the CAR functor as a corollary to all the propositions above.

\begin{prop}
The CAR flows associated with isometric representations $V^1$ and $V^2$ are cocycle conjugate if and only if $V^1$ and $V^2$ are conjugate. In that case the CAR flows are actually conjugate.
\end{prop}

\begin{proof} Thanks to the discussions in the previous section and Proposition \ref{H}, we only have to prove that the gauge cocycles act transitively on the set of units of a CAR flow $\beta$ associated with an isometric representation $V$. Define $$ \mathfrak{A}_0 = \{\Log_\Omega(u):  u \in \mathfrak{U}_\Omega(\beta) \}\subseteq  \mathfrak{A}(V); ~~~K_0 = \overline{span}\{V_x h_y: h \in  \mathfrak{A}_0, x,y \in P\}\subseteq K.$$ 
Define $W: \Gamma_a(K_0)\otimes \Gamma_a(K_0^\perp) \mapsto \Gamma_a(K)$ as the unitary extension of 
$$ (\xi_1 \wedge\xi_2  \cdots \wedge \xi_{m}) \otimes (\eta_1 \wedge\eta_2   \cdots \wedge \eta_{n}) \mapsto    \xi_1 \wedge\xi_2   \cdots \wedge \xi_{m} \wedge  \eta_1 \wedge  \eta_2 \cdots \wedge  \eta_{n}$$ for $\xi_i \in \Ker(K_0)$, $\eta_j \in K_0^\perp$, $i\in [m], j \in [n]$. 

It can be easily seen that $V$ restricts to an isometric representation on $K_0$, which we denote by $V_0$. Denote the CCR flow and CAR flow associated with $V^0$ by $\alpha^0$ and $\beta^0$ respectively.  Since the units of $\beta$ can be  constructed from the additive cocycles in $ \mathfrak{A}_0 $ through the $\Exp$ map in  \cite[Proposition 5.10]{MS1}, it is easy to see that the units of $\beta$ are indeed in $\Gamma_a(K_0)$ as a subspace of $\Gamma_a(K)$. So any unit of $\beta$ is of the form $W(u_0\otimes \Omega)$, where $u_0$ is a unit for $\beta^0$. Conversely for any unit $u_0$ of  $\beta^0$, $W(u_0\otimes \Omega)$ is a unit for $\beta$. (Here $\Omega$ is the vacuum vector in $\Gamma_a(K^\perp)$.)

Since $\beta^0$ is type I, it is  cocycle conjugate to $\alpha^0$, thanks to Proposition \ref{I}.   Given a unit $u_0 \in   \mathfrak{U}_\Omega(\beta^0)$, with $u =W(u_0\otimes \Omega) \in  \mathfrak{U}_\Omega(\beta)$, we can choose the gauge cocycle $\{W(\Log_\Omega(u)_x): x \in P\}$ of $\alpha^0$, which maps the vacuum unit of $\alpha^0$ to  (scalar multiples)  of $\{e(\Log_\Omega(u)_x): x \in P\}$.  Now using the cocycle conjugacy, we get a gauge cocycle $\tilde{U}^u$ of $\beta^0$, which maps the vacuum unit of $\beta^0$ to $u_0$. Now define $U^u_x =  W(\tilde{U}^u_x \otimes 1)W^*$. We claim $U^u = \{U^u_x: x \in P\}$ is the required gauge cocycle of $\beta$ which maps $\Omega_a$ to $u$.

We first claim that $\alpha_x(W(T\otimes 1)W^*) = W(\alpha^0_x(T) \otimes 1)W^*$ for all $T \in B(\Gamma_a(K_0))$. It is easy to see $W(a(f)\otimes 1)W^* = a(f)$ for  all $f \in K_0$. Hence $$ \alpha_x(a(f))=a(Vf) =a(V^0f) =W(a(V^0 f)\otimes 1)W^* =W(\alpha^0_x(a(f))\otimes 1)W^*~~ \forall f \in K_0.$$ Since $\{a(f): f  \in K_0\}$ generates $B(\Gamma_a(K_0))$, the claim follows. Now \begin{align*} &  U^u_x\alpha_x(U^u_y)  =  W(\tilde{U}^u_x \otimes 1)W^*~ \alpha_x (W(\tilde{U}^u_y \otimes 1)W^*)
 =  W(\tilde{U}^u_x \otimes 1)W^*~ W(\alpha^0_x (\tilde{U}^u_y) \otimes 1)W^*\\
 &  =  W((\tilde{U}^u_x (\alpha^0_x( \tilde{U}^u_y)) \otimes 1 )W^*
 = W(\tilde{U}^u_{x+y} \otimes 1)W^* = U^u_{x+y}
\end{align*}
The proof is over
\end{proof}

\section{Non-cocycle-conjugacy between CCR flows and CAR flows}

We assume a technical condition on the isometric representation $V$ and we prove that the associated CCR flow is not cocycle conjugate to the associated CAR flow. For $x \in P$ we set $K_x =\Ker(V_x^*)$ and use this notation hereafter. 

\begin{defn} An isometric representation $V$ of $P$ is called proper if there exists $x,y \in P$ satisfying \begin{align} \label{xy}  V_x(K_y) \cap (V_x(K_y)\cap V_y(K_x))^\perp \neq \{0\}; V_y(K_x) \cap (V_x(K_y)\cap V_y(K_x))^\perp  \neq \{0\}\\ (V_x(K_y) \cap (V_x(K_y)\cap V_y(K_x))^\perp ) ~ \perp ~ (V_y(K_x) \cap (V_x(K_y)\cap V_y(K_x))^\perp )\nonumber  \end{align}
\end{defn}

Notice that $1-$parameter isometric representations can not satisfy the conditions (\ref{xy}).  Let  $P =\R_+^2$, $V_{s,t} =S_s\oplus S_t$ on $L^2((0,\infty),\k_1)\oplus L^2((0,\infty),\k_2)$. In this basic example it is a routine verification to check that the above conditions are satisfied for any $x=(s,t)$ and $y =(s',t')$ with $s\leq s'$ and $t\geq t'$. 

A vast family of examples of isometric representations are constructed as follows. Let $A \subseteq \R^{d}$ be a nonempty closed subset. Let $A$ be a $P$-module, that is $A+P \subseteq A$.   Let $K:= L^{2}(A,\k)$. For $x \in P$. The  isometric representation $S^A= \{S^A_{x}\}_{x \in P}$ is  defined by \begin{equation}
S^A_{x}(f)(y):=\begin{cases}
 f(y-x)  & \mbox{ if
} y -x \in A,\cr
   &\cr
    0 &  \mbox{ if } y-x \notin A.
         \end{cases}
\end{equation}
The conditions (\ref{xy}) means the existence of $x,y\in P$ satisfying \begin{align*} & \left((A+x+y)^c\cap (A+x)\right) \setminus \left((A+x+y)^c\cap (A+y)\right)\neq \emptyset\\ &   \left((A+x+y)^c\cap (A+y)\right) \setminus \left((A+x+y)^c\cap (A+x)\right)\neq \emptyset. \end{align*}  
 When $A$ is a region  bounded below and in the left, and unbounded in the right and above, the above conditions  (\ref{xy}) can be seen to hold.  
 
 \begin{prop}\label{non-conj} Let $V$ be an isometric representation satisfying conditions $(\ref{xy})$. Then the associated CCR flow $\alpha$  is not cocycle conjugate to the CAR flow $\beta$. 
 \end{prop}
 
 \begin{proof} Suppose let $\beta_x = Ad(U)Ad(U_x) \alpha_x Ad(U^*)$ where $\{U_x:x \in P\}$ be a cocycle for $\alpha$ and $U$ is a unitary operator between $\Gamma_s(K)$ to  $\Gamma_a(K)$.  
 
 Set $\m_s=  B(\Gamma_s(K))$ and $\m_a=  B(\Gamma_a(K))$  Then for any two $x,y\in P$ we have  \begin{align*} Ad(U U_{x+y})\left(\alpha_x(\m_s)\bigcap \alpha_{x+y}(\m_s)'\right) & = Ad(U U_{x}\alpha_x(U_y) ) (\alpha_x(\m_s) )\bigcap  Ad(U U_{x+y}) (\alpha_{x+y}(\m_s)' )\\
 & = Ad(U U_{x}) (\alpha_x(\m_s) ) \bigcap Ad(U U_{x+y}) ( \alpha_{x+y}(\m_s))'\\
 &= \beta_x(\m_a)  \bigcap\beta_{x+y}(\m_a)'.
 \end{align*}
 Similarly we have  $Ad(U U_{x+y})\left(\alpha_y(\m_s)\bigcap \alpha_{x+y}(\m_s)'\right) = \beta_y(\m_a)  \bigcap\beta_{x+y}(\m_a)'.$
 
 Now it is clear if $\left(\alpha_x(\m_s)\cap \alpha_{x+y}(\m_s)'\right)  \bigcap \left( \alpha_x(\m_s) \cap \alpha_y(\m_s)\cap \alpha_{x+y}(\m_s)' \right)'$ and \linebreak  $\left(\alpha_y(\m_s)\cap \alpha_{x+y}(\m_s)'\right)  \bigcap \left(\alpha_x(\m_s)  \cap \alpha_y(\m_s)\cap \alpha_{x+y}(\m_s)' \right)'$ are non-trivial and commute with each other, then 
 $\left(\beta_x(\m_a)\cap \beta_{x+y}(\m_a)'\right)  \bigcap \left( \beta_x(\m_a)  \cap \beta_y(\m_a)\cap \beta_{x+y}(\m_a)' \right)'$  and \linebreak   $\left(\beta_y(\m_a)\cap \beta_{x+y}(\m_a)'\right)  \bigcap \left( \beta_x(\m_a)  \cap \beta_y(\m_a)\cap \beta_{x+y}(\m_a)' \right)'$ are also non-trivial and should commute with each other.
 Now to prove the non-cocycle-conjugacy between $\alpha$ and $\beta$, we show, using (\ref{xy}),  that all the above algebras are non-trivial, and the commutation relations hold for the CCR flow $\alpha$ but not for the CAR flow $\beta$.
 
 By the very definition of CCR flows  in Example \ref{ex-CCR-CAR}, we can identify $\alpha_{x+y}(\m_s)'$ with $B(\Gamma_s(K_{x+y}))$  through the natural isomorphism defined in (\ref{UCCR}). Now $\alpha_x(\m_s)\cap \alpha_{x+y}(\m_s)'$ can be identified with $B(\Gamma_s(V_xK_y))\subseteq B(\Gamma_s(K_{x+y}))$ and $\alpha_y(\m_s)\cap \alpha_{x+y}(\m_s)'$  with $B(\Gamma_s(V_yK_x))\subseteq B(\Gamma_s(K_{x+y}))$, using the natural embeddings. Since $B(\cdot)$ of symmetric Fock spaces of orthogonal subspaces commute, by the condition (\ref{xy}), the algebras mentioned are non-trivial and commute, as claimed for $\alpha$. 
 
 For $\beta$ there is a twist.  The definition of CAR in Example \ref{ex-CCR-CAR} imply, through the unitary map defined in (\ref{UCAR}), that $\beta_{x+y}(\m_a)' = \{a^{\#} (\xi) R_{x+y} : \xi \in K_{x+y}\}'',$ where $a^{\#} (\xi)$ denote either the Fermionic creation operator or the annihilation operator, and $R_{x+y}$ is defined by the unitary extension of  
 $$V_{x+y} \eta_1 \wedge  \cdots \wedge V_{x+y} \eta_{n} \wedge \xi_1 \wedge    \cdots \wedge \xi_{m} \mapsto (-1)^n  V_{x+y} \eta_1 \wedge  \cdots \wedge V_{x+y} \eta_{n} \wedge \xi_1 \wedge   \cdots \wedge \xi_{m} ,$$ for $\eta_1, \cdots \eta_n \in K$ and $\xi_1, \cdots \xi_m \in K_{x+y}$. Also we have  \begin{align*}  & \beta_x(\m_s)\cap \beta_{x+y}(\m_s)'= \{a^{\#} (\xi) R_{x+y} : \xi \in V_xK_{y}\}'';\\  & \beta_y(\m_s)\cap \beta_{x+y}(\m_s)' =\{a^{\#} (\xi) R_{x+y} : \xi \in V_yK_{x}\}''.\end{align*} Now through the same analysis we get \begin{align*}
 & \left(\beta_x(\m_a)\cap \beta_{x+y}(\m_a)'\right)  \bigcap \left( \beta_x(\m_a)  \cap \beta_y(\m_a)\cap \beta_{x+y}(\m_a)' \right)'\\ & =  \{a^{\#} (\xi) R_{x+y} R_{x,y}: \xi \in V_x(K_y) \cap (V_x(K_y)\cap V_y(K_x))^\perp\}'';\\
  & \left(\beta_y(\m_a)\cap \beta_{x+y}(\m_a)'\right)  \bigcap \left( \beta_x(\m_a)  \cap \beta_y(\m_a)\cap \beta_{x+y}(\m_a)' \right)'\\ & =  \{a^{\#} (\xi) R_{x+y} R_{x,y}: \xi \in V_y(K_x) \cap (V_x(K_y)\cap V_y(K_x))^\perp\}'', 
 \end{align*} where $R_{x,y}$ is defined by the unitary extension of  
 $$\eta_1 \wedge  \cdots \wedge \eta_{n} \wedge \xi_1 \wedge    \cdots \wedge \xi_{m} \mapsto (-1)^n  \eta_1 \wedge  \cdots \wedge  \eta_{n} \wedge \xi_1 \wedge   \cdots \wedge \xi_{m} ,$$ for $\eta_1, \cdots \eta_n \in V_x(K_y)\cap V_y(K_x)$ and $\xi_1, \cdots \xi_m \in (V_x(K_y)\cap V_y(K_x))^\perp$. 
 
 Now the algebras are clearly non-trivial, thanks to condition (\ref{xy}).  $R_{x+y} R_{x,y}$ commutes with $a^{\#} (\xi)$ for any $\xi \in V_x(K_y) \cap (V_x(K_y)\cap V_y(K_x))^\perp$ and for any  $\xi \in V_y(K_x) \cap (V_x(K_y)\cap V_y(K_x))^\perp$. So the operators $a^{\#} (\xi) R_{x+y} R_{x,y}$ and $a^{\#} (\eta) R_{x+y} R_{x,y}$ anticommute for $\xi \in V_x(K_y) \cap (V_x(K_y)\cap V_y(K_x))^\perp$ and $\eta \in V_y(K_x) \cap (V_x(K_y)\cap V_y(K_x))^\perp$. The proof of non-cocycle-conjugacy is finally over. 
 \end{proof}

 \end{document}